\documentclass[11pt]{article}
\usepackage{enumerate}
\usepackage{amssymb,a4wide,latexsym,makeidx,epsfig,fleqn}
\usepackage{amsthm}
\usepackage{amsmath}
\usepackage{enumerate}
\usepackage{graphicx}
\usepackage{float}
\usepackage{tikz}
\usepackage[colorlinks=true, linkcolor=blue, citecolor=red, urlcolor=blue]{hyperref}
\allowdisplaybreaks[4]
\newtheorem{theorem}{Theorem}[section]

\newtheorem{lemma}[theorem]{Lemma}

\newtheorem{problem}{Problem}

\begin{document}
\textwidth 150mm \textheight 225mm
\title{Spectral radius conditions for edge-disjoint spanning trees in $(k+c)$-edge-connected graphs\thanks{Supported by the National Natural Science Foundation of China (No. 12271439).}}
\author{{Yongbin Gao$^{a,b}$, Ligong Wang$^{a,b,}$\footnote{Corresponding author.}}\\
{\small $^a$ School of Mathematics and Statistics, Northwestern
Polytechnical University,}\\ {\small  Xi'an, Shaanxi 710129,
P.R. China.}\\
{\small $^b$ Xi'an-Budapest Joint Research Center for Combinatorics, Northwestern
Polytechnical University,}\\
{\small Xi'an, Shaanxi 710129,
P.R. China. }\\
{\small E-mail: gybmath@163.com, lgwangmath@163.com} }
\date{}
\maketitle
\begin{center}
\begin{minipage}{120mm}
\vskip 0.3cm
\begin{center}
{\small {\bf Abstract}}
\end{center}
{\small Let $\tau(G)$ denote the spanning tree packing number of a graph $G$. Recently, Zhang and Fan [J. Graph Theory 112 (2) (2026) 128--144] posed the problem of finding a tight spectral radius condition for an $m$-edge-connected graph $G$ to guarantee $\tau(G)\ge k$ for $k+1\le m\le 2k-1$. They solved the cases $m=k$ and $k=2, m=3$. In this paper, we study this problem for all $m=k+c$, where $1\le c\le k-1$. For $1\le c\le k-2$, we obtain a tight spectral radius condition for a $(k+c)$-edge-connected graph to contain $k$ edge-disjoint spanning trees. We also obtain a tight spectral radius condition for $(2k-1)$-edge-connected graphs. In both cases, we give graph families containing all extremal graphs, and the graphs with maximum spectral radius in these families serve as the corresponding extremal graphs. Each graph in these families consists of a large clique and a small remaining part, with certain restrictions on the edges inside the small part and between the two parts. Moreover, for the case $m=k+1$, we further determine the unique extremal graph.

\vskip 0.1in \noindent {\bf Keywords}: \ Edge-disjoint spanning trees, Edge connectivity, Spectral radius}
\end{minipage}
\end{center}

\section{Introduction}

In this paper, we consider only finite, undirected and simple graphs. Let $G$ be a graph with vertex set $V(G)$ and edge set $E(G)$. We denote the order of $G$ by $n=|V(G)|$. For a vertex $v\in V(G)$, let $N_G(v)$ and $d_G(v)$ denote the neighborhood and the degree of $v$ in $G$, respectively. For a subset $X\subseteq V(G)$, let $N_X(v)=N_G(v)\cap X$ and $d_X(v)=|N_X(v)|$, and let $e_G(X)$ denote the number of edges in the subgraph of $G$ induced by $X$. Let $\delta(G)$ denote the minimum degree of $G$. For any two disjoint subsets $X,Y \subseteq V(G)$, let $E_G(X,Y)$ be the set of edges with one end in $X$ and the other end in $Y$, and let $e_G(X,Y)=|E_G(X,Y)|$. If $X=\{v\}$, then we simply write $E_G(v,Y)$ and $e_G(v,Y)$. When the graph $G$ is clear from the context, we omit the subscript $G$. Let $\kappa'(G)$ denote the edge connectivity of $G$, and let $\tau(G)$ denote the spanning tree packing number of $G$, that is, the maximum number of edge-disjoint spanning trees contained in $G$.

The spanning tree packing number is closely related to edge connectivity. Indeed, the theorem of Nash-Williams and Tutte (see Theorem \ref{thm:nash_williams} in Section 2) implies that $\tau(G)\ge \left\lfloor \frac{\kappa'(G)}{2}\right\rfloor$, and hence every $2k$-edge-connected graph contains $k$ edge-disjoint spanning trees. While the existence of $k$ edge-disjoint spanning trees always implies $k$-edge-connectivity, the converse does not hold in general. It is therefore natural to ask what additional conditions force a graph with given edge connectivity to contain $k$ edge-disjoint spanning trees. This question was also studied by Lai and Li \cite{Lai2019}.

\begin{problem}[\cite{Lai2019}]\label{prob1}
	Find conditions on a graph $G$ with $k \le \kappa'(G) \le 2k-1$ that guarantee $\tau(G) \ge k$.
\end{problem}

Recall that an edge-cut $X$ of a graph $G$ is called \emph{$r$-essential} if $G-X$ has two components, each of which contains at least $r$ edges. A graph $G$ is called \emph{$r$-essentially $s$-edge-connected} if it has no $r$-essential edge-cut of size less than $s$. In particular, if $r=1$, we call it essential. Lai and Li \cite{Lai2019} proved that every $m$-edge-connected, essentially $h$-edge-connected graph contains $k$ edge-disjoint spanning trees whenever $m\ge k+1$ and $h\ge \frac{m^2}{m-k}-2$. They also showed that this condition is sharp.

Gu, Liu and Yu \cite{Gu2023} continued this line of research from the perspective of 2-essential edge connectivity. They proved that every $m$-edge-connected, $2$-essentially $h$-edge-connected graph $G$ contains $k$ edge-disjoint spanning trees, provided that $k+1\le m\le 2k-1$, $G$ is neither $K_5$ nor a fat triangle of multiplicity less than $k$, and
$h$ is at least a certain explicit function of $m$ and $k$. Moreover, they showed that every $3$-edge-connected, essentially $5$-edge-connected, and $2$-essentially $8$-edge-connected graph contains two edge-disjoint spanning trees.

Recently, Zhang and Fan \cite{Zhang2026} studied Problem \ref{prob1} from the perspective of spectral radius. They established a tight spectral radius condition ensuring $\tau(G)\ge k$ for $k$-edge-connected graphs and characterized the extremal graph. They also obtained a spectral condition for $k$-edge-connected graphs with fixed minimum degree $\delta(G)\ge 4k$. In addition, they further studied analogous results on 3-edge-connected graphs to guarantee $\tau(G)\ge 2$, which implies comprehensive spectral arguments on $\tau(G)\ge 2$ for $\kappa'(G)<4$.

Another motivation comes from a problem proposed by Seymour in a private communication to Cioab\v{a} (see \cite{Cioaba2012}).

\begin{problem}[\cite{Cioaba2012}]\label{prob2}
	Let $G$ be a connected graph. Determine the relationship between $\tau(G)$ and the eigenvalues of $G$.
\end{problem}

Let $A(G)$ be the adjacency matrix of a graph $G$ of order $n$, and let $\lambda_1(G) \ge \lambda_2(G) \ge \cdots \ge \lambda_n(G)$ be the eigenvalues of $A(G)$. The largest eigenvalue $\lambda_1(G)$ is called the spectral radius of $G$, denoted by $\rho(G)$. Motivated by Problem \ref{prob2}, Cioab\v{a} and Wong \cite{Cioaba2012} initiated the study of the relationship between $\tau(G)$ and adjacency eigenvalues. They obtained sufficient conditions in terms of $\lambda_2(G)$ for $d$-regular graphs to satisfy $\tau(G) \ge k$ for $k=2,3$, and proposed a conjecture for $4 \le k \le \lfloor \frac{d}{2} \rfloor$. This conjecture was later extended by Gu et al. \cite{Gu2016} to graphs with minimum degree $\delta$, and was subsequently confirmed by Liu et al. \cite{Liu2014,2Liu2014}. More precisely, they proved that if $G$ has minimum degree $\delta \ge 2k$ and $\lambda_2(G) < \delta - \frac{2k-1}{\delta+1}$, then $\tau(G) \ge k$. From the perspective of the spectral radius, Fan et al. \cite{Fan2023} proved that if $k \ge 2$ and $G$ is a connected graph with minimum degree $\delta \ge 2k$ and order $n \ge 2\delta+3$, then $\rho(G) \ge \rho(B_{n,\delta+1}^{k-1})$ implies $\tau(G) \ge k$ unless $G \cong B_{n,\delta+1}^{k-1}$, where $B_{n,\delta+1}^{k-1}$ is obtained from $K_{\delta+1} \cup K_{n-\delta-1}$ by adding $k-1$ edges joining one vertex in $K_{\delta+1}$ to $k-1$ vertices in $K_{n-\delta-1}$. More results on the relationship between eigenvalues of graphs and the spanning tree packing number can be found in \cite{Cai2026,Cioaba2022,Duan2020,Fan2025,Gao2026,Hu2023,Liu2019}.

Recently, motivated by both Problems \ref{prob1} and \ref{prob2}, Zhang and Fan \cite{Zhang2026} studied spectral radius conditions for edge-disjoint spanning trees under prescribed edge connectivity. They obtained tight spectral radius conditions for $k$-edge-connected graphs, and further posed the following problem.

\begin{problem}[\cite{Zhang2026}]\label{prob3}
	For $k \ge 2$ and $k+1 \le m \le 2k-1$, let $G$ be an $m$-edge-connected graph. Find a tight spectral radius condition that guarantees $\tau(G) \ge k$, and characterize the corresponding extremal graphs.
\end{problem}

As a first step toward Problem \ref{prob3}, Zhang and Fan \cite{Zhang2026} studied the case $k=2$ and $m=3$. Let $F_1$ be the graph defined in \cite{Zhang2026} and shown in Figure \ref{fig:F1}.

\begin{theorem}[\cite{Zhang2026}]\label{thm1}
	Let $G$ be a $3$-edge-connected graph of order $n \ge 10$. If $\rho(G) \ge \rho(F_1)$, then $G$ contains two edge-disjoint spanning trees unless $G \cong F_1$.
\end{theorem}

Motivated by the above problems and results, we study Problem \ref{prob3} for all cases $m=k+c$ with $1\le c\le k-1$. We establish tight spectral radius conditions for $(k+c)$-edge-connected graphs to guarantee $\tau(G) \ge k$. In both cases, the extremal graphs are described by explicit extremal families. Moreover, in the case $c=1$, we further determine the unique extremal graph.

We now introduce two graph families in our main results.

Let $1\le c\le k-2$. Let $\mathcal{H}_{n,k}^{c}$ be the class of all $(k+c)$-edge-connected graphs $H$ of order $n$ for which there exists a partition $V(H)=U\cup T$ such that $|U|=n-2c-2$, $|T|=2c+2$, $H[U]\cong K_{n-2c-2}$, $2c^2+2c+1\le e_H(T)\le 2c^2+3c+1$ and $e_H(T,U)=k(2c+2)-1-e_H(T)$. Let $H_{n,k}^{c,*}$ be a graph in $\mathcal{H}_{n,k}^{c}$ with maximum spectral radius.

For the case $c=k-1$, let $\mathcal{H}_{n,k}$ be the class of all $(2k-1)$-edge-connected graphs $H$ of order $n$ for which there exists a partition $V(H)=U\cup T$ such that $|U|=n-2k-1$, $|T|=2k+1$, $H[U]\cong K_{n-2k-1}$, $e_H(T)=k(2k-1)$ and $e_H(T,U)=2k-1$. Let $H^*_{n,k}$ be a graph in $\mathcal{H}_{n,k}$ with maximum spectral radius.

We can now present our main results.

\begin{theorem}\label{thm:k+c}
	Let $k, c$ be integers with $1 \le c \le k-2$, and let $G$ be a $(k+c)$-edge-connected graph of order $n\ge (c+2)k-c^2+c+2$. If $\rho(G)\ge \rho(H_{n,k}^{c,*})$, then $\tau(G)\ge k$ unless $G$ is a graph in $\mathcal{H}_{n,k}^{c}$ with maximum spectral radius, that is, $G\in \mathcal{H}_{n,k}^{c}$ and $\rho(G)=\rho(H_{n,k}^{c,*})$.
\end{theorem}

\begin{theorem}\label{thm:k-1}
	Let $k\ge 2$ be a integer, and let $G$ be a $(2k-1)$-edge-connected graph of order $n\ge 4k+2$. If $\rho(G)\ge \rho(H^*_{n,k})$, then $\tau(G)\ge k$ unless $G$ is a graph in $\mathcal{H}_{n,k}$ with maximum spectral radius, that is, $G\in \mathcal{H}_{n,k}$ and $\rho(G)=\rho(H^*_{n,k})$.
\end{theorem}

Theorem \ref{thm:k+c} gives a tight spectral radius condition in the range $1\le c\le k-2$. For the case $c=1$, we determine the unique extremal graph. We now describe this graph explicitly, as shown in Figure \ref{fig:Fnk}. Let $k\ge 3$ and $n\ge 3k+2$. Let $U=\{u_1,u_2,\ldots,u_{n-4}\}$ and $T=\{v_1,v_2,v_3,v_4\}$ be disjoint vertex sets. Let $F_{n,k}$ be the graph with vertex set $U\cup T$ such that $F_{n,k}[U]\cong K_{n-4}$, $F_{n,k}[T]\cong K_4-e$ with missing edge $v_1v_2$, and $E(T,U)=\{v_i u_j : i\in\{1,2\},\ 1\le j\le k-1\}\cup\{v_i u_j : i\in\{3,4\},\ 1\le j\le k-2\}$.

\begin{figure}[htbp]
	\centering
	\tikzset{
		vertex/.style={circle, fill=black, inner sep=0pt, minimum size=7pt},
		edot/.style={circle, fill=black, inner sep=0pt, minimum size=2.2pt}
	}
	
	\begin{minipage}[t]{0.48\textwidth}
		\centering
		\begin{tikzpicture}[line width=1pt]
			\draw (0, 4.464) circle (1cm);
			\node[font=\large] at (0, 4.464) {$K_{n-5}$};
			
			\node[vertex] (a) at (0, 3.464) {};
			\node[vertex] (b) at (0, 2.314) {};
			\node[vertex] (c) at (-1, 0.577) {};
			\node[vertex] (d) at (1, 0.577) {};
			\node[vertex] (e) at (-2, 0) {};
			\node[vertex] (f) at (2, 0) {};
			
			\draw (a) -- (b);
			\draw (b) -- (c);
			\draw (b) -- (d);
			\draw (c) -- (d);
			\draw (a) -- (e);
			\draw (a) -- (f);
			\draw (c) -- (e);
			\draw (d) -- (f);
			\draw (e) -- (f);
		\end{tikzpicture}
		\caption{The graph $F_1$}
		\label{fig:F1}
	\end{minipage}
	\hfill
	\begin{minipage}[t]{0.48\textwidth}
		\centering
		\begin{tikzpicture}[line width=1pt]
			\node[vertex, label=below left:$v_4$] (v4) at (-1.3, 0) {};
			\node[vertex, label=below right:$v_3$] (v3) at (1.3, 0) {};
			\node[vertex, label=above left:$v_1$] (v1) at (-1.8, 1.3) {};
			\node[vertex, label=above right:$v_2$] (v2) at (1.8, 1.3) {};
			
			\draw (0, 4.0) circle (1.5cm);
			\node[font=\large] at (0, 4.5) {$K_{n-4}$};
			
			\node[vertex, label=above:$u_1$] (u1) at (-0.9, 3.10) {};
			\node[vertex, label=above:$u_{k-2}$] (u2) at (0.15, 3.10) {};
			\node[vertex, label=above:$u_{k-1}$] (u3) at (0.95, 3.10) {};
			
			\node[edot] at (-0.55, 3.10) {};
			\node[edot] at (-0.42, 3.10) {};
			\node[edot] at (-0.29, 3.10) {};
			
			\draw (v1) -- (v3);
			\draw (v1) -- (v4);
			\draw (v2) -- (v3);
			\draw (v2) -- (v4);
			\draw (v3) -- (v4);
			
			\draw (v1) -- (u1);
			\draw (v1) -- (u2);
			\draw (v1) -- (u3);
			
			\draw (v2) -- (u1);
			\draw (v2) -- (u2);
			\draw (v2) -- (u3);
			
			\draw (v3) -- (u1);
			\draw (v3) -- (u2);
			
			\draw (v4) -- (u1);
			\draw (v4) -- (u2);
		\end{tikzpicture}
		\caption{The graph $F_{n,k}$ with vertex labels}
		\label{fig:Fnk}
	\end{minipage}
\end{figure}

\begin{theorem}\label{thm:c1}
	Let $k\ge 3$ be a integer, and let $G$ is a $(k+1)$-edge-connected graph of order $n\ge 3k+2$. If $\rho(G)\ge \rho(F_{n,k})$, then $\tau(G)\ge k$ unless $G\cong F_{n,k}$.
\end{theorem}

The rest of this paper is organized as follows. In Section 2, we present some known results that will be used in the proofs of our main results. In Section 3, we prove Theorems \ref{thm:k+c} and \ref{thm:k-1}. In Section 4, we prove Theorem \ref{thm:c1}. In Section 5, we give some concluding remarks.

\section{Preliminaries}

In this section, we present some known results to be used in the proofs of our main results.

We begin with the classical theorem of Nash-Williams \cite{Nash-Williams1961} and Tutte \cite{Tutte1961}, which characterizes graphs containing $k$ edge-disjoint spanning trees. For any partition $\pi$ of $V(G)$, let $e_G(\pi)$ denote the number of edges of $G$ with ends in distinct parts of $\pi$.

\begin{theorem}[\cite{Nash-Williams1961,Tutte1961}]\label{thm:nash_williams}
	A connected graph $G$ has $k$ edge-disjoint spanning trees if and only if
	\[
	e_G(\pi) \ge k(t-1)
	\]
	for every partition $\pi$ of $V(G)$, where $t$ is the number of parts in the partition $\pi$.
\end{theorem}

We next present some useful properties of the spectral radius and the Perron vector of a graph.

\begin{lemma}\label{lem:spectral_subgraph}
	Let $G$ be a connected graph and let $H$ be a subgraph of $G$. Then $\rho(H) \le \rho(G)$, with equality if and only if $H \cong G$.
\end{lemma}

\begin{lemma}[\cite{Wu2005}]\label{lem:perron_edge_shifting}
	Let $G$ be a connected graph, and let $u,v$ be two vertices of $G$. Suppose that there exists a set of vertices $\{v_1,v_2,\dots,v_s\} \subseteq N_G(v) \setminus N_G(u)$ with $s \ge 1$. Let $G^*$ be the graph obtained from $G$ by deleting the edges $vv_i$ and adding the edges $uv_i$ for all $1 \le i \le s$. Let $\mathbf{x}$ be the Perron vector of $A(G)$, and let $\mathbf{x}_w$ denote the coordinate of $\mathbf{x}$ corresponding to the vertex $w \in V(G)$. If $\mathbf{x}_u \ge \mathbf{x}_v$, then $\rho(G) < \rho(G^*)$.
\end{lemma}

The following lemma gives a sharp upper bound on the spectral radius.

\begin{lemma}[\cite{Hong2001,Nikiforov2002}]\label{lem:spectral_upper_bound}
	Let $G$ be a graph of order $n$ and size $m$ with minimum degree $\delta \ge 1$. Then
	\[
	\rho(G) \le \frac{\delta-1}{2} + \sqrt{2m-\delta n+\frac{(\delta+1)^2}{4}},
	\]
	with equality if and only if $G$ is either a $\delta$-regular graph or a bidegreed graph in which each vertex has degree either $\delta$ or $n-1$.
\end{lemma}

\begin{lemma}[\cite{Hong2001,Nikiforov2002}]\label{lem:function_decreasing}
	Let $p$ and $q$ be nonnegative integers with $2q \le p(p-1)$. Then the function
	\[
	f(x)=\frac{x-1}{2}+\sqrt{2q-px+\frac{(x+1)^2}{4}}
	\]
	is decreasing on $0 \le x \le p-1$.
\end{lemma}

We also need the following combinatorial upper bounds.

\begin{lemma}[\cite{Wei2022}]\label{lem:combinatorial_bound}
	Let $a_1,a_2,\dots,a_l$ be integers with $l \ge 2$ and $0 \le a_1 \le a_2 \le \cdots \le a_l$. Let $f(x_1,x_2,\dots,x_l)=\sum_{i=1}^l \binom{x_i}{2}$, where $\sum_{i=1}^l x_i=n$ and $x_i \ge a_i$ for all $1 \le i \le l$. Then
	\[
	f(x_1,x_2,\dots,x_l) \le \binom{n-\sum_{i=1}^{l-1} a_i}{2}+\sum_{i=1}^{l-1}\binom{a_i}{2}.
	\]
	Moreover, the equality holds if and only if $x_i=a_i$ for all $1 \le i \le l-1$ and $x_l=n-\sum_{i=1}^{l-1} a_i$.
\end{lemma}

\begin{lemma}\label{lem:sum_bound}
	Let $p,n,l$ be integers with $n\ge 2p$ and $3\le l\le p+1$. Let $x_1,x_2,\dots,x_l$ be positive integers such that $x_1\ge x_2\ge \cdots\ge x_l$, $\sum_{i=1}^l x_i=n$, and $\sum_{i=2}^l x_i\ge p$. Then
	\[
	\sum_{i=1}^l \binom{x_i}{2}\le \binom{n-p}{2}+\binom{p-l+2}{2}.
	\]
	Moreover, equality holds if and only if $x_1=n-p$, $x_2=p-l+2$ and $x_3=\cdots=x_l=1$.
\end{lemma}

\begin{proof}
	Choose a feasible sequence $x_1,x_2,\ldots,x_l$ for which $\sum_{i=1}^l\binom{x_i}{2}$ is maximum, and let $r=\sum_{i=2}^l x_i$. We first claim that $r=p$. Suppose to the contrary that $r\ge p+1$. Since $l\le p+1$, we have $l-1\le p$. By the Pigeonhole Principle, there exists some $x_j$ with $2\le j\le l$ such that $x_j\ge 2$.
	
	Let $x_1'=x_1+1$, $x_j'=x_j-1$, and let all other terms remain unchanged. After rearranging $x_2',\ldots,x_l'$ in non-increasing order if necessary, we obtain a new sequence $x_1'\ge x_2'\ge \cdots\ge x_l'$. Clearly, $\sum_{i=1}^l x_i'=n$ and $\sum_{i=2}^l x_i'=r-1\ge p$. Hence this new sequence is still feasible. Moreover,
	\begin{align*}
		\sum_{i=1}^l \binom{x_i'}{2}-\sum_{i=1}^l \binom{x_i}{2}
		&=\binom{x_1+1}{2}-\binom{x_1}{2}
		+\binom{x_j-1}{2}-\binom{x_j}{2}\\
		&=x_1-x_j+1\ge 1,
	\end{align*}
	which contradicts the maximality of $x_1,x_2,\ldots,x_l$. Therefore, $r=p$, and hence $x_1=n-p$.
	
	Now $\sum_{i=2}^{l} x_i=p$ and $x_2\ge x_3\ge\cdots\ge x_l\ge 1$. By Lemma \ref{lem:combinatorial_bound}, we have
	\[
	\sum_{i=2}^l\binom{x_i}{2}\le \binom{p-(l-2)}{2}=\binom{p-l+2}{2}.
	\]
	Hence,
	\[
	\sum_{i=1}^l\binom{x_i}{2}
	\le \binom{n-p}{2}+\binom{p-l+2}{2}.
	\]
	
	Equality holds only if $r=p$ and equality holds in Lemma \ref{lem:combinatorial_bound} for $x_2,\ldots,x_l$. Hence equality holds only if $x_1=n-p$, $x_2=p-l+2$ and $x_3=\cdots=x_l=1$. Conversely, since $l\le p+1$, we have $p-l+2\ge 1$. Since $n\ge 2p$ and $l\ge 3$, we have $n-p\ge p\ge p-l+2$. Hence the sequence $x_1=n-p$, $x_2=p-l+2$ and $x_3=\cdots=x_l=1$ is feasible and attains equality. This completes the proof.
\end{proof}

We also need the following theorem, which will be applied to closed convex quadrilaterals in our proofs below.

\begin{theorem}[\cite{Niculescu2025}]\label{thm:convex_max_vertex}
	If $f$ is a continuous convex function on a compact convex subset $K$ of $\mathbb{R}^N$, then $f$ attains a global maximum at an extreme point.
\end{theorem}

\section{Proofs of Theorems \ref{thm:k+c} and \ref{thm:k-1}}

We now present the proof of Theorem \ref{thm:k+c}.

\begin{proof}[Proof of Theorem \ref{thm:k+c}]
	Since $1\le c\le k-2$, we have $k\ge c+2$. Let $G$ be a graph with maximum spectral radius among all $(k+c)$-edge-connected graphs of order $n\ge (c+2)k-c^2+c+2$ with at most $k-1$ edge-disjoint spanning trees. It suffices to show that $G\in \mathcal{H}_{n,k}^{c}$ and $\rho(G)=\rho(H_{n,k}^{c,*})$.
	
	Recall that every graph in $\mathcal{H}_{n,k}^{c}$ has a partition $V(H)=U\cup T$ with $|T|=2c+2$ and $e_H(T,U)=k(2c+2)-1-e_H(T)$. Since the partition $\{U\}\cup\{\{v\}:v\in T\}$ satisfies $e_H(T,U)+e_H(T)=k(2c+2)-1<k(2c+3-1)$, every graph in $\mathcal{H}_{n,k}^{c}$ has at most $k-1$ edge-disjoint spanning trees by Theorem \ref{thm:nash_williams}. Therefore, $H_{n,k}^{c,*}$ belongs to the same graph class from which $G$ was chosen, and hence $\rho(G)\ge \rho(H_{n,k}^{c,*})$.
	
	Since $G$ has at most $k-1$ edge-disjoint spanning trees, Theorem \ref{thm:nash_williams} implies that there exists a partition $\pi$ of $V(G)$ into $t=t_1+t_2$ parts, where $t_1$ and $t_2$ denote the numbers of trivial and nontrivial parts, respectively, such that
	\begin{equation}\label{eq:partition_edges}
		e_G(\pi)\le k(t-1)-1=kt-k-1.
	\end{equation}
	
	Let the trivial and nontrivial parts of $\pi$ be $\{v_1\}, \{v_2\}, \ldots, \{v_{t_1}\}$ and $U_1, U_2, \ldots, U_{t_2}$, respectively. Since $G$ is $(k+c)$-edge-connected, we have $e_G(v_i,V(G)\setminus \{v_i\})\ge k+c$ for each $i \in \{1, 2, \ldots, t_1\}$, and $e_G(U_j,V(G)\setminus U_j)\ge k+c$ for each $j \in \{1, 2, \ldots, t_2\}$. By the handshaking lemma, 
	\[
	2e_G(\pi)=\sum_{i=1}^{t_1}e_G(v_i,V(G)\setminus \{v_i\})+\sum_{j=1}^{t_2}e_G(U_j,V(G)\setminus U_j)\ge (k+c)t_1+(k+c)t_2=(k+c)t.
	\]
	Thus, $(k+c)t\le 2e_G(\pi)\le 2(kt-k-1)$, which implies $t\ge \frac{2k+2}{k-c}=2+\frac{2c+2}{k-c}$. Since $1\le c\le k-2$ and $t$ is an integer, we have $t\ge 3$.
	
	Since $K_{n-2c-2}$ is a proper subgraph of $H_{n,k}^{c,*}$, Lemma \ref{lem:spectral_subgraph} implies that $\rho(G)\ge \rho(H_{n,k}^{c,*})>\rho(K_{n-2c-2})=n-2c-3$. Moreover, since $\delta(G)\ge \kappa'(G)\ge k+c$, Lemmas \ref{lem:spectral_upper_bound} and \ref{lem:function_decreasing} imply that
	\begin{align*}
		n-2c-3<\rho(G)
		&\le \frac{\delta(G)-1}{2}+\sqrt{2e(G)-\delta(G)n+\frac{(\delta(G)+1)^2}{4}} \\
		&\le \frac{k+c-1}{2}+\sqrt{2e(G)-(k+c)n+\frac{(k+c+1)^2}{4}}.
	\end{align*}
	Hence we have
	\begin{equation}\label{eq:edges_lower_bound}
		e(G)>\frac{1}{2}n^2-\frac{4c+5}{2}n+(c+1)k+3c^2+6c+3.
	\end{equation}
	
	\noindent \textbf{Claim 1.} \textit{$t_2\ge 1$.} 
	
	\noindent \textit{Proof.} Suppose to the contrary that $t_2=0$. Then all parts of $\pi$ are trivial, so $t_1=t=n$ and $e(G)=e_G(\pi)$. Set
	\begin{align*}
		\phi(n)
		&=\frac{1}{2}n^2-\frac{4c+5}{2}n+(c+1)k+3c^2+6c+3-(kn-k-1)\\
		&=\frac{1}{2}n^2-\frac{4c+2k+5}{2}n+(c+2)k+3c^2+6c+4.
	\end{align*}
	Then $\phi(n)$ is a quadratic function in $n$ which opens upward, and its axis of symmetry is $n^*=\frac{4c+2k+5}{2}$. Since $1\le c\le k-2$, we have
	\begin{align*}
		(c+2)k-c^2+c+2-\frac{4c+2k+5}{2}
		&=\frac{(2c+2)k-2c^2-2c-1}{2}\\
		&\ge \frac{(2c+2)(c+2)-2c^2-2c-1}{2}\\
		&=\frac{4c+3}{2}>0.
	\end{align*}
	Hence $\phi(n)$ is increasing for $n\ge (c+2)k-c^2+c+2$. Thus,
	\[
	\phi(n)\ge \phi\left((c+2)k-c^2+c+2\right)=\frac{1}{2}\left(c(k-c-1)-1\right)\left(c(k-c-1)+2k-2\right)\ge 0.
	\]
	
	By Inequality \eqref{eq:partition_edges}, we have
	\[
	e(G)\le kn-k-1 \le \frac{1}{2}n^2-\frac{4c+5}{2}n+(c+1)k+3c^2+6c+3,
	\]
	which contradicts Inequality \eqref{eq:edges_lower_bound}. Thus, $t_2\ge 1$ holds. \qed
	
	\noindent \textbf{Claim 2.} \textit{$t\ge 2c+3$.}
	
	\noindent \textit{Proof.} Suppose to the contrary that $3\le t\le 2c+2$. Without loss of generality, assume that the parts of $\pi$ are ordered as $|X_1|\ge |X_2|\ge \cdots \ge |X_t|$. Let $r=\sum_{i=2}^t |X_i|$. We claim that $r\ge 2c+3$. Otherwise, $t-1\le r\le 2c+2$.
	
	Since $\delta(G)\ge k+c$, by Lemma \ref{lem:combinatorial_bound}, we have
	\begin{align*}
		r(k+c) \le \sum_{i=2}^{t}\sum_{v\in X_i}d_G(v)
		&=2\sum_{i=2}^t e_G(X_i)+2\sum_{2\le i<j\le t} e_G(X_i, X_j)
		+e_G(X_1,V(G)\setminus X_1) \\
		&=\sum_{i=2}^t e_G(X_i)
		+\left(\sum_{i=2}^t e_G(X_i)+\sum_{2\le i<j\le t} e_G(X_i, X_j)\right)\\
		&\quad+\left(\sum_{2\le i<j\le t} e_G(X_i, X_j)+e_G(X_1,V(G)\setminus X_1)\right) \\
		&=\sum_{i=2}^t e_G(X_i)
		+e_G\left(\bigcup_{i=2}^t X_i\right)
		+e_G(\pi) \\
		&\le \binom{r-(t-2)}{2}+\binom{r}{2}+kt-k-1.
	\end{align*}
	Rearranging the above inequality, we obtain
	\begin{equation}\label{eq:f(rt)}
		0\le r^2-(k+c+t-1)r+\frac{1}{2}t^2+\frac{2k-3}{2}t-k:=f(r,t).
	\end{equation}
	The Hessian matrix $H_f$ of $f(r,t)$ is
	\[
	H_f=
	\begin{pmatrix}
		\frac{\partial^2 f}{\partial r^2}
		&
		\frac{\partial^2 f}{\partial r\partial t}
		\\[4pt]
		\frac{\partial^2 f}{\partial t\partial r}
		&
		\frac{\partial^2 f}{\partial t^2}
	\end{pmatrix}
	=
	\begin{pmatrix}
		2&-1\\
		-1&1
	\end{pmatrix}.
	\]
	Since $2>0$ and $\det(H_f)=1>0$, the matrix $H_f$ is positive definite. Hence the function $f(r,t)$ is convex on $\mathbb{R}^2$. By the previous assumption and discussion, the variables $r$ and $t$ satisfy $r\le 2c+2$, $r\ge t-1$ and $3\le t\le 2c+2$. These inequalities define a closed convex quadrilateral $Q_1$ with vertices $(2,3)$, $(2c+2,3)$, $(2c+2,2c+2)$ and $(2c+1,2c+2)$. By Theorem \ref{thm:convex_max_vertex} and $1\le c\le k-2$, we have
	\begin{align*}
		f(r,t)\le \max_{(r,t)\in Q_1}f(r,t)
		&= \max\{f(2,3),f(2c+2,3),f(2c+2,2c+2),f(2c+1,2c+2)\}\\
		&=\max\{-2c,2c(c+1-k),c+1-k,-1\}<0,
	\end{align*}
	which contradicts Inequality \eqref{eq:f(rt)}. Therefore, the claim $r\ge 2c+3$ holds.
	
	Since $n\ge (c+2)k-c^2+c+2\ge (c+2)^2-c^2+c+2=5c+6>2(2c+3)$
	and $3\le t\le 2c+2$, by Lemma \ref{lem:sum_bound}, we have
	\[
	\sum_{i=1}^t e_G(X_i)\le \sum_{i=1}^{t} \binom{|X_i|}{2}\le  \binom{n-(2c+3)}{2}+\binom{2c+3-t+2}{2}.
	\]
	Combining this with Inequality \eqref{eq:partition_edges}, we obtain
	\begin{align*}
		e(G)
		&=\sum_{i=1}^t e_G(X_i)+e_G(\pi)\\
		&\le \binom{n-2c-3}{2}+\binom{2c-t+5}{2}+kt-k-1\\
		&=\frac{1}{2}t^2+\left(k-2c-\frac{9}{2}\right)t
		+\frac{1}{2}n^2-\frac{4c+7}{2}n
		+4c^2+16c-k+15\\
		&:=\psi(t).
	\end{align*}
	
	Since $\psi(t)$ is a quadratic function in $t$ which opens upward and $3\le t\le 2c+2$, we have
	\begin{align*}
		e(G) \le \psi(t)
		&\le \max\{\psi(3),\psi(2c+2)\}\\
		&=\max\bigg\{
		\frac{1}{2}n^2-\frac{4c+7}{2}n+4c^2+10c+2k+6,\\
		&\qquad
		\frac{1}{2}n^2-\frac{4c+7}{2}n+2c^2+7c+2ck+k+8
		\bigg\}\\
		&=\frac{1}{2}n^2-\frac{4c+7}{2}n+2c^2+7c+2ck+k+8 \\
		&< \frac{1}{2}n^2-\frac{4c+5}{2}n+(c+1)k+3c^2+6c+3,
	\end{align*}
	which contradicts Inequality \eqref{eq:edges_lower_bound}. Thus, $t\ge 2c+3$ holds. \qed

	\noindent \textbf{Claim 3.} \textit{$t_1\ge 2c+2$.}
	
	\noindent \textit{Proof.} Suppose to the contrary that $0\le t_1\le 2c+1$. By Lemma \ref{lem:combinatorial_bound}, we have $e(G)\le g(t_1,t_2)$, where
	\begin{align*}
		e(G)&\le\binom{n-t_1-2(t_2-1)}{2}+(t_2-1)\binom{2}{2}+k(t_1+t_2)-(k+1)\\
		&=\frac{1}{2}(n-t_1-2t_2+2)(n-t_1-2t_2+1)+kt_1+(k+1)t_2-k-2\\
		&=2t_2^2-(2n-2t_1-k+2)t_2+\frac{1}{2}(n-t_1+2)(n-t_1+1)+k(t_1-1)-2\\
		&=:g(t_1,t_2).
	\end{align*}
	
	The Hessian matrix $H_g$ of $g(t_1,t_2)$ is
	\[
	H_g=
	\begin{pmatrix}
		\frac{\partial^2 g}{\partial t_1^2}&\frac{\partial^2 g}{\partial t_1\partial t_2}\\
		\frac{\partial^2 g}{\partial t_2\partial t_1}&\frac{\partial^2 g}{\partial t_2^2}
	\end{pmatrix}
	=
	\begin{pmatrix}
		1&2\\
		2&4
	\end{pmatrix}.
	\]
	Since $1>0$ and $\det(H_g)=0$, the matrix $H_g$ is positive semi-definite. Hence $g(t_1,t_2)$ is convex on $\mathbb{R}^2$. Since each nontrivial part contains at least $2$ vertices, we have $n\ge t_1+2t_2$. By Claim 2, we have $t_1+t_2\ge 2c+3$. Together with $0\le t_1\le 2c+1$, these inequalities define a closed convex quadrilateral $Q_2$ with vertices $(0,2c+3)$, $(0,\frac{n}{2})$, $(2c+1,\frac{n-2c-1}{2})$ and $(2c+1,2)$.
	
	By Theorem \ref{thm:convex_max_vertex} and $n\ge (c+2)k-c^2+c+2$, we have
	\begin{align*}
		e(G)
		&\le \max_{(t_1,t_2)\in Q_2} g(t_1,t_2)\\
		&= \max\left\{
		g(0,2c+3),\,
		g\left(0,\frac{n}{2}\right),\,
		g\left(2c+1,\frac{n-2c-1}{2}\right),\,
		g(2c+1,2)
		\right\}\\
		&=\max\bigg\{
		\frac{1}{2}n^2-\frac{8c+9}{2}n+2ck+2k+8c^2+20c+11,\,
		\frac{1}{2}(k+1)n-k-1,\\
		&\qquad	\frac{1}{2}(k+1)n+ck-c-\frac{k}{2}-\frac{3}{2},\,
		\frac{1}{2}n^2-\frac{4c+7}{2}n+2ck+2k+2c^2+7c+6
		\bigg\}\\
		&=\frac{1}{2}n^2-\frac{4c+7}{2}n+2ck+2k+2c^2+7c+6\\
		&< \frac{1}{2}n^2-\frac{4c+5}{2}n+(c+1)k+3c^2+6c+3,
	\end{align*}
	which contradicts Inequality \eqref{eq:edges_lower_bound}. Thus, $t_1\ge 2c+2$ holds. \qed
	
	\noindent \textbf{Claim 4.} \textit{$t=2c+3$.}
	
	\noindent \textit{Proof.} Suppose to the contrary that $t\ge 2c+4$. Then we still have $e(G)\le g(t_1,t_2)$. By the previous analysis and assumptions, we have $t_2\ge 1$, $t_1\ge 2c+2$, $t_1+2t_2\le n$ and $t_1+t_2\ge 2c+4$. These inequalities define a closed convex quadrilateral $Q_3$ with vertices $(2c+2,2)$, $(2c+2,\frac{n-2c-2}{2})$, $(n-2,1)$ and $(2c+3,1)$.
	
	By Theorem \ref{thm:convex_max_vertex} and $n\ge (c+2)k-c^2+c+2$, we have
	\begin{align*}
		e(G)
		&\le \max_{(t_1,t_2)\in Q_3} g(t_1,t_2)\\
		&= \max\left\{
		g(2c+2,2),\,
		g\left(2c+2,\frac{n-2c-2}{2}\right),\,
		g(n-2,1),\,
		g(2c+3,1)
		\right\}\\
		&=\max\bigg\{
		\frac{1}{2}n^2-\frac{4c+9}{2}n+2ck+3k+2c^2+9c+10,\,
		\frac{1}{2}(k+1)n+ck-c-2,\\
		&\qquad
		k(n-2),\,
		\frac{1}{2}n^2-\frac{4c+7}{2}n+2ck+3k+2c^2+7c+5
		\bigg\}\\
		&=\frac{1}{2}n^2-\frac{4c+7}{2}n+2ck+3k+2c^2+7c+5\\
		&\le \frac{1}{2}n^2-\frac{4c+5}{2}n+(c+1)k+3c^2+6c+3,
	\end{align*}
	which contradicts Inequality \eqref{eq:edges_lower_bound}. Thus, $t=2c+3$ holds. \qed 
	
	Since $t_1\ge 2c+2$ and $t_2\ge 1$, we must have $t_1=2c+2$ and $t_2=1$. Then the unique nontrivial part is $U_1$. Let $T_1=\{v_1,v_2,\dots,v_{2c+2}\}$ be the union of trivial parts.
	
	We claim that $G[U_1]\cong K_{n-2c-2}$. Otherwise, there exists a missing edge inside $U_1$. Adding this edge does not change $e_G(\pi)$, so the same partition $\pi$ still satisfies Inequality \eqref{eq:partition_edges}. Hence, by Theorem \ref{thm:nash_williams}, the resulting graph still has at most $k-1$ edge-disjoint spanning trees. Moreover, adding an edge cannot decrease the edge connectivity. Therefore, the resulting graph remains in the same graph class, while its spectral radius is strictly larger by Lemma \ref{lem:spectral_subgraph}, contradicting the maximality of $\rho(G)$. Hence $G[U_1]\cong K_{n-2c-2}$.
	
	We claim that $e_G(\pi)=(2c+2)k-1$. Suppose that $e_G(\pi)\le (2c+2)k-2$. Then there exists a missing edge whose ends lie in distinct parts of $\pi$. Adding this edge increases $e_G(\pi)$ by one, and hence the same partition $\pi$ still satisfies $e_G(\pi)\le (2c+2)k-1<k(2c+2)$. By Theorem \ref{thm:nash_williams}, the resulting graph still has at most $k-1$ edge-disjoint spanning trees. Moreover, adding an edge cannot decrease the edge connectivity. Therefore, the resulting graph remains in the same graph class, while its spectral radius is strictly larger by Lemma \ref{lem:spectral_subgraph}, contradicting the maximality of $\rho(G)$. Thus $e_G(\pi)=(2c+2)k-1$. Since $e_G(\pi)=e_G(T_1,U_1)+e_G(T_1)$, we have $e_G(T_1,U_1)+e_G(T_1)=(2c+2)k-1$.
	
	Since $\delta(G)\ge k+c$, for each $i\in\{1,2,\cdots,2c+2\}$ we have $d_G(v_i)\ge k+c$. Summing over all vertices in $T_1$, we obtain
	\[
	\sum_{i=1}^{2c+2} d_G(v_i)=2e_G(T_1)+e_G(T_1,U_1)\ge (2c+2)(k+c).
	\]
    Hence $e_G(T_1)\ge (2c+2)(k+c)-\left((2c+2)k-1 \right)= 2c^2+2c+1$. Since $e_G(T_1)\le \binom{2c+2}{2}=2c^2+3c+1$, $e_G(T_1)$ is constrained to 
	\[
	2c^2+2c+1\le e_G(T_1)\le 2c^2+3c+1.
	\]
	Hence we obtain $G\in \mathcal{H}_{n,k}^{c}$.
	
	By the maximality of $\rho(G)$, we have $\rho(G)\ge \rho(H_{n,k}^{c,*})$. On the other hand, since $G\in \mathcal{H}_{n,k}^{c}$ and $H_{n,k}^{c,*}$ has maximum spectral radius in $\mathcal{H}_{n,k}^{c}$, we have $\rho(G)\le \rho(H_{n,k}^{c,*})$. Hence $\rho(G)=\rho(H_{n,k}^{c,*})$. Therefore, $G$ is a graph in $\mathcal{H}_{n,k}^{c}$ with maximum spectral radius. This completes the proof.
	
\end{proof}

We next present the proof of Theorem \ref{thm:k-1}.

\begin{proof}[Proof of Theorem \ref{thm:k-1}]
	Let $G$ be a graph with maximum spectral radius among all $(2k-1)$-edge-connected graphs of order $n\ge 4k+2$ with at most $k-1$ edge-disjoint spanning trees. It suffices to show that $G\in \mathcal{H}_{n,k}$ and $\rho(G)=\rho(H^*_{n,k})$.
	
	Recal that every graph in $\mathcal{H}_{n,k}$ has a partition $V(H)=U\cup T$ with $|T|=2k+1$, $e_H(T)=k(2k-1)$ and $e_H(T,U)=2k-1$. Since the partition $\{U\}\cup\{\{v\}:v\in T\}$ satisfies $e_H(T,U)+e_H(T)=k(2k+1)-1<k(2k+2-1)$, every graph in $\mathcal{H}_{n,k}$ has at most $k-1$ edge-disjoint spanning trees by Theorem \ref{thm:nash_williams}. Therefore, $H^*_{n,k}$ belongs to the same graph class from which $G$ was chosen, and hence $\rho(G)\ge \rho(H_{n,k}^*)$
	
	Since $G$ has at most $k-1$ edge-disjoint spanning trees, Theorem \ref{thm:nash_williams} implies that there exists a partition $\pi$ of $V(G)$ into $t=t_1+t_2$ parts, where $t_1$ and $t_2$ denote the numbers of trivial and nontrivial parts, respectively, such that
	\begin{equation}\label{eq:partition_edges_endpoint}
		e_G(\pi)\le k(t-1)-1=kt-k-1.
	\end{equation}
	
	Let the trivial and nontrivial parts of $\pi$ be $\{v_1\}, \{v_2\}, \ldots, \{v_{t_1}\}$ and $U_1, U_2, \ldots, U_{t_2}$, respectively. Since $G$ is $(2k-1)$-edge-connected, we have $e_G(v_i,V(G)\setminus \{v_i\})\ge 2k-1$ for each $i\in\{1,2,\ldots,t_1\}$, and $e_G(U_j,V(G)\setminus U_j)\ge 2k-1$ for each $j\in\{1,2,\ldots,t_2\}$. By the handshaking lemma,
	\[
	2e_G(\pi)=\sum_{i=1}^{t_1}e_G(v_i,V(G)\setminus \{v_i\})+\sum_{j=1}^{t_2}e_G(U_j,V(G)\setminus U_j)\ge (2k-1)t.
	\]
	Combining this with Inequality \eqref{eq:partition_edges_endpoint}, we have $(2k-1)t\le 2(kt-k-1)$, and hence $t\ge 2k+2$.
	
	Since $K_{n-2k-1}$ is a proper subgraph of $H^*_{n,k}$, Lemma \ref{lem:spectral_subgraph} implies that $\rho(G)\ge \rho(H^*_{n,k})>\rho(K_{n-2k-1})=n-2k-2$. Moreover, since $\delta(G)\ge \kappa'(G)\ge 2k-1$, Lemmas \ref{lem:spectral_upper_bound} and \ref{lem:function_decreasing} imply that
	\begin{align*}
		n-2k-2<\rho(G)
		&\le \frac{\delta(G)-1}{2}+\sqrt{2e(G)-\delta(G)n+\frac{(\delta(G)+1)^2}{4}} \\
		&\le k-1+\sqrt{2e(G)-(2k-1)n+k^2}.
	\end{align*}
	Hence we have
	\begin{equation}\label{eq:edges_lower_bound_endpoint}
		e(G)>\frac{1}{2}n^2-\frac{4k+3}{2}n+4k^2+3k+\frac{1}{2}.
	\end{equation}
	
	\noindent \textbf{Claim 1.} \textit{$t_2\ge 1$.}
	
	\noindent \textit{Proof.} Suppose that $t_2=0$. Then all parts of $\pi$ are trivial, so $t_1=t=n$ and $e(G)=e_G(\pi)$. Set
	\begin{align*}
		\phi(n)
		&=\frac{1}{2}n^2-\frac{4k+3}{2}n+4k^2+3k+\frac{1}{2}-(kn-k-1)\\
		&=\frac{1}{2}n^2-\frac{6k+3}{2}n+4k^2+4k+\frac{3}{2}.
	\end{align*}
	Then $\phi(n)$ is a quadratic function in $n$ which opens upward, and its axis of symmetry is $n^*=\frac{6k+3}{2}$. Since $n\ge 4k+2$, we have $4k+2-\frac{6k+3}{2}=\frac{2k+1}{2}>0$. Hence $\phi(n)$ is increasing for $n\ge 4k+2$. Therefore, $\phi(n)\ge \phi(4k+2)=\frac{1}{2}>0$. By Inequality \eqref{eq:partition_edges_endpoint}, we have
	\[
	e(G)\le kn-k-1<\frac{1}{2}n^2-\frac{4k+3}{2}n+4k^2+3k+\frac{1}{2},
	\]
	which contradicts Inequality \eqref{eq:edges_lower_bound_endpoint}. Thus, $t_2\ge 1$ holds. \qed
	
	\noindent \textbf{Claim 2.} \textit{$t_1\ge 2k+1$.}
	
	\noindent \textit{Proof.} Suppose to the contrary that $0\le t_1\le 2k$. By Lemma \ref{lem:combinatorial_bound}, we have
	\begin{align*}
		e(G)
		&\le \binom{n-t_1-2(t_2-1)}{2}+(t_2-1)\binom{2}{2}+k(t_1+t_2)-(k+1)\\
		&=\frac{1}{2}(n-t_1-2t_2+2)(n-t_1-2t_2+1)+kt_1+(k+1)t_2-k-2\\
		&=2t_2^2-(2n-2t_1-k+2)t_2+\frac{1}{2}(n-t_1+2)(n-t_1+1)+k(t_1-1)-2\\
		&=:g(t_1,t_2).
	\end{align*}
	The Hessian matrix $H_g$ of $g(t_1,t_2)$ is
	\[
	H_g=
	\begin{pmatrix}
		\frac{\partial^2 g}{\partial t_1^2}&\frac{\partial^2 g}{\partial t_1\partial t_2}\\
		\frac{\partial^2 g}{\partial t_2\partial t_1}&\frac{\partial^2 g}{\partial t_2^2}
	\end{pmatrix}
	=
	\begin{pmatrix}
		1&2\\
		2&4
	\end{pmatrix}.
	\]
	Since $1\ge 0$ and $\det(H_g)=0$, the matrix $H_g$ is positive semi-definite. Hence $g(t_1,t_2)$ is convex on $\mathbb{R}^2$. Since each nontrivial part contains at least $2$ vertices, we have $t_1+2t_2\le n$. Moreover, by the above discussion, $t_1+t_2\ge 2k+2$. Together with $0\le t_1\le 2k$, these inequalities define a closed convex quadrilateral $Q_1$ with vertices $(0,2k+2)$, $(0,\frac{n}{2})$, $(2k,\frac{n-2k}{2})$ and $(2k,2)$.
	
	By Theorem \ref{thm:convex_max_vertex} and $n\ge 4k+2$, we have
	\begin{align*}
		e(G)
		&\le \max_{(t_1,t_2)\in Q_1} g(t_1,t_2)\\
		&= \max\left\{
		g(0,2k+2),\,
		g\left(0,\frac{n}{2}\right),\,
		g\left(2k,\frac{n-2k}{2}\right),\,
		g(2k,2)
		\right\}\\
		&=\max\bigg\{
		\frac{1}{2}n^2-\frac{8k+5}{2}n+10k^2+13k+3,\,
		\frac{1}{2}(k+1)n-k-1,\\
		&\qquad
		\frac{1}{2}(k+1)n+k^2-2k-1,\,
		\frac{1}{2}n^2-\frac{4k+5}{2}n+4k^2+6k+3
		\bigg\}\\
		&=\frac{1}{2}n^2-\frac{4k+5}{2}n+4k^2+6k+3\\
		&<\frac{1}{2}n^2-\frac{4k+3}{2}n+4k^2+3k+\frac{1}{2},
	\end{align*}
	which contradicts Inequality \eqref{eq:edges_lower_bound_endpoint}. Thus, $t_1\ge 2k+1$ holds. \qed
	
	\noindent \textbf{Claim 3.} \textit{$t=2k+2$.}
	
	\noindent \textit{Proof.} Suppose to the contrary that $t\ge 2k+3$. Then we still have $e(G)\le g(t_1,t_2)$. By the previous analysis and assumptions, we have $t_2\ge 1$, $t_1\ge 2k+1$, $t_1+2t_2\le n$ and $t_1+t_2\ge 2k+3$. These inequalities define a closed convex quadrilateral $Q_3$ with vertices $(2k+1,2)$, $(2k+1,\frac{n-2k-1}{2})$, $(n-2,1)$ and $(2k+2,1)$.
	
	By Theorem \ref{thm:convex_max_vertex} and $n\ge 4k+2$, we have
	\begin{align*}
		e(G)
		&\le \max_{(t_1,t_2)\in Q_3} g(t_1,t_2)\\
		&= \max\left\{
		g(2k+1,2),\,
		g\left(2k+1,\frac{n-2k-1}{2}\right),\,
		g(n-2,1),\,
		g(2k+2,1)
		\right\}\\
		&=\max\bigg\{
		\frac{1}{2}n^2-\frac{4k+7}{2}n+4k^2+9k+6,\,
		\frac{1}{2}(k+1)n+k^2-\frac{3}{2}k-\frac{3}{2},\\
		&\qquad
		k(n-2),\,
		\frac{1}{2}n^2-\frac{4k+5}{2}n+4k^2+7k+2
		\bigg\}\\
		&=\frac{1}{2}n^2-\frac{4k+5}{2}n+4k^2+7k+2\\
		&< \frac{1}{2}n^2-\frac{4k+3}{2}n+4k^2+3k+\frac{1}{2},
	\end{align*}
	which contradicts Inequality \eqref{eq:edges_lower_bound_endpoint}. Thus, $t=2k+2$ holds. \qed
	
	Since $t_1\ge 2k+1$ and $t_2\ge 1$, we must have $t_1=2k+1$ and $t_2=1$. Then the unique nontrivial part is $U_1$. Let $T_1=\{v_1,v_2,\dots,v_{2k+1}\}$ be the union of trivial parts.
	
	We claim that $G[U_1]\cong K_{n-2k-1}$. Otherwise, there exists a missing edge inside $U_1$. Adding this edge does not change $e_G(\pi)$, so the same partition $\pi$ still satisfies Inequality \eqref{eq:partition_edges_endpoint}. Hence, by Theorem \ref{thm:nash_williams}, the resulting graph still has at most $k-1$ edge-disjoint spanning trees. Moreover, adding an edge cannot decrease the edge connectivity. Therefore, the resulting graph remains in the same graph class, while its spectral radius is strictly larger by Lemma \ref{lem:spectral_subgraph}, contradicting the maximality of $\rho(G)$. Hence $G[U_1]\cong K_{n-2k-1}$.
	
	We claim that $e_G(\pi)=k(2k+1)-1$. Suppose that $e_G(\pi)\le k(2k+1)-2$. Then there exists a missing edge whose ends lie in distinct parts of $\pi$. Adding this edge increases $e_G(\pi)$ by one, and hence the same partition $\pi$ still satisfies $e_G(\pi)\le k(2k+1)-1<k(2k+1)$. By Theorem \ref{thm:nash_williams}, the resulting graph still has at most $k-1$ edge-disjoint spanning trees. Moreover, adding an edge cannot decrease the edge connectivity. Therefore, the resulting graph remains in the same graph class, while its spectral radius is strictly larger by Lemma \ref{lem:spectral_subgraph}, contradicting the maximality of $\rho(G)$. Thus $e_G(\pi)=k(2k+1)-1$. Since $e_G(\pi)=e_G(T_1,U_1)+e_G(T_1)$, we have $e_G(T_1,U_1)+e_G(T_1)=k(2k+1)-1$.
	
	Since $\delta(G)\ge 2k-1$, for each $i\in\{1,2,\cdots,2k+1\}$ we have $d_G(v_i)\ge 2k-1$. Summing over all vertices in $T_1$, we obtain
	\[
	\sum_{i=1}^{2k+1} d_G(v_i)=2e_G(T_1)+e_G(T_1,U_1)\ge (2k+1)(2k-1),
	\]
    which yields
	\[
	e_G(T_1)\ge (2k+1)(2k-1)-\left(k(2k-1)-1\right)=2k^2-k.
	\]
	Hence
	\[
	e_G(T_1,U_1)=k(2k+1)-1-e_G(T_1)\le 2k-1.
	\]
	On the other hand, since $G$ is $(2k-1)$-edge-connected, we have $e_G(T_1,U_1)\ge 2k-1$. Hence $e_G(T_1,U_1)=2k-1$ and $e_G(T_1)=k(2k-1)$. Hence we obtain $G\in \mathcal{H}_{n,k}$.
	
	By the maximality of $\rho(G)$, we have $\rho(G)\ge \rho(H^*_{n,k})$. On the other hand, since $G\in \mathcal{H}_{n,k}$ and $H^*_{n,k}$ has maximum spectral radius in $\mathcal{H}_{n,k}$, we have $\rho(G)\le \rho(H^*_{n,k})$. Hence $\rho(G)=\rho(H^*_{n,k})$. Therefore, $G$ is a graph in $\mathcal{H}_{n,k}$ with maximum spectral radius. This completes the proof.
	
\end{proof}

\section{Proof of Theorem \ref{thm:c1}}

Let $F'_{n,k}$ be the graph obtained from $F_{n,k}$ by deleting the edge $v_2u_{k-1}$ and adding the edge $v_1v_2$, as shown in Figure \ref{fig:Fnk&F'nk}. Before proving Theorem \ref{thm:c1}, we compare the spectral radii of $F_{n,k}$ and $F'_{n,k}$.

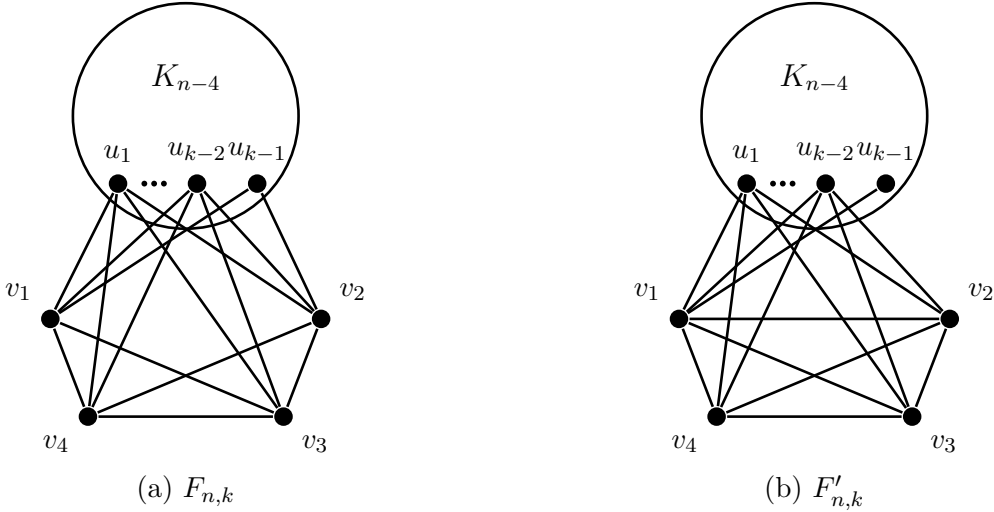
\begin{figure}[htbp]
	\centering
	\tikzset{
		vertex/.style={circle, fill=black, inner sep=0pt, minimum size=7pt},
		edot/.style={circle, fill=black, inner sep=0pt, minimum size=2.2pt}
	}
	
	\begin{minipage}[b]{0.48\textwidth}
		\centering
		\begin{tikzpicture}[line width=1pt]
			\node[vertex, label=below left:$v_4$] (v4) at (-1.3, 0) {};
			\node[vertex, label=below right:$v_3$] (v3) at 	(1.3, 0) {};
			\node[vertex, label=above left:$v_1$] (v1) at 	(-1.8, 1.3) {};
			\node[vertex, label=above right:$v_2$] (v2) at 	(1.8, 1.3) {};
			
			\draw (0, 4.0) circle (1.5cm);
			\node[font=\large] at (0, 4.5) {$K_{n-4}$};
			
			\node[vertex, label=above:$u_1$] (u1) at (-0.9, 3.10) {};
			\node[vertex, label=above:$u_{k-2}$] (u2) at (0.15, 3.10) {};
			\node[vertex, label=above:$u_{k-1}$] (u3) at (0.95, 3.10) {};
			
			\node[edot] at (-0.55, 3.10) {};
			\node[edot] at (-0.42, 3.10) {};
			\node[edot] at (-0.29, 3.10) {};
			
			\draw (v1) -- (v3);
			\draw (v1) -- (v4);
			\draw (v2) -- (v3);
			\draw (v2) -- (v4);
			\draw (v3) -- (v4);
			
			\draw (v1) -- (u1);
			\draw (v1) -- (u2);
			\draw (v1) -- (u3);
			
			\draw (v2) -- (u1);
			\draw (v2) -- (u2);
			\draw (v2) -- (u3);
			
			\draw (v3) -- (u1);
			\draw (v3) -- (u2);
			
			\draw (v4) -- (u1);
			\draw (v4) -- (u2);
		\end{tikzpicture}
		
		(a) $F_{n,k}$
	\end{minipage}
	\hfill
	\begin{minipage}[b]{0.48\textwidth}
		\centering
		\begin{tikzpicture}[line width=1pt]
			\node[vertex, label=below left:$v_4$] (v4) at (-1.3, 0) {};
			\node[vertex, label=below right:$v_3$] (v3) at (1.3, 0) {};
			\node[vertex, label=above left:$v_1$] (v1) at (-1.8, 1.3) {};
			\node[vertex, label=above right:$v_2$] (v2) at (1.8, 1.3) {};
			
			\draw (0, 4.0) circle (1.5cm);
			\node[font=\large] at (0, 4.5) {$K_{n-4}$};
			
			\node[vertex, label=above:$u_1$] (u1) at (-0.9, 3.10) {};
			\node[vertex, label=above:$u_{k-2}$] (u2) at (0.15, 3.10) {};
			\node[vertex, label=above:$u_{k-1}$] (u3) at (0.95, 3.10) {};
			
			\node[edot] at (-0.55, 3.10) {};
			\node[edot] at (-0.42, 3.10) {};
			\node[edot] at (-0.29, 3.10) {};
			
			\draw (v1) -- (v2);
			\draw (v1) -- (v3);
			\draw (v1) -- (v4);
			\draw (v2) -- (v3);
			\draw (v2) -- (v4);
			\draw (v3) -- (v4);
			
			\draw (v1) -- (u1);
			\draw (v1) -- (u2);
			\draw (v1) -- (u3);
			
			\draw (v2) -- (u1);
			\draw (v2) -- (u2);
			
			\draw (v3) -- (u1);
			\draw (v3) -- (u2);
			
			\draw (v4) -- (u1);
			\draw (v4) -- (u2);
		\end{tikzpicture}
		
		(b) $F'_{n,k}$
	\end{minipage}
	\caption{The graphs $F_{n,k}$ and $F'_{n,k}$ with vertex labels}
	\label{fig:Fnk&F'nk}
\end{figure}

\begin{lemma}\label{lem:F2_F4_compare}
	For $n \ge 3k+2$ and $k \ge 3$, we have $\rho(F_{n,k}) > \rho(F'_{n,k})$.
\end{lemma}

\begin{proof}
	Let $\mathbf{x}$ be the Perron vector of $A(F'_{n,k})$, and let $\mathbf{x}_w$ denote the entry of $\mathbf{x}$ corresponding to a vertex $w \in V(F'_{n,k})$. Since
	\[
	F_{n,k}=F'_{n,k}-v_1v_2+v_2u_{k-1},
	\]
	we may apply Lemma \ref{lem:perron_edge_shifting} with $u=u_{k-1}$, $v=v_1$, and $\{v_2\}\subseteq N_{F'_{n,k}}(v_1)\setminus N_{F'_{n,k}}(u_{k-1})$. Thus, it suffices to show that $\mathbf{x}_{u_{k-1}} \ge \mathbf{x}_{v_1}$. For simplicity, we write $\rho=\rho(F'_{n,k})$.
	
	The eigenequations at $v_2$ and $v_3$ are
	\begin{align*}
		\rho \mathbf{x}_{v_2} &= \sum_{l=1}^{k-2} \mathbf{x}_{u_l} + \mathbf{x}_{v_1} + \mathbf{x}_{v_3} + \mathbf{x}_{v_4}, \\
		\rho \mathbf{x}_{v_3} &= \sum_{l=1}^{k-2} \mathbf{x}_{u_l} + \mathbf{x}_{v_1} + \mathbf{x}_{v_2} + \mathbf{x}_{v_4}.
	\end{align*}
	Subtracting these two equations, we obtain
	\[
	(\rho+1)(\mathbf{x}_{v_3}-\mathbf{x}_{v_2})=0.
	\]
	Since $\rho>0$, it follows that $\mathbf{x}_{v_2}=\mathbf{x}_{v_3}$. By symmetry, we have $\mathbf{x}_{v_2} = \mathbf{x}_{v_3} = \mathbf{x}_{v_4}$ and $\mathbf{x}_{u_k} = \mathbf{x}_{u_{k+1}} = \dots = \mathbf{x}_{u_{n-4}}$. Let $\mathbf{x}_v = \mathbf{x}_{v_2} = \mathbf{x}_{v_3} = \mathbf{x}_{v_4}$ and $\mathbf{x}_u = \mathbf{x}_{u_k} = \mathbf{x}_{u_{k+1}} = \dots = \mathbf{x}_{u_{n-4}}$.
	
	The eigenequations simplify to
	\begin{align*}
		\rho \mathbf{x}_v &= \sum_{l=1}^{k-2} \mathbf{x}_{u_l} + \mathbf{x}_{v_1} + 2\mathbf{x}_v
		\implies \sum_{l=1}^{k-2} \mathbf{x}_{u_l} + \mathbf{x}_{v_1} = (\rho - 2)\mathbf{x}_v, \\
		\rho \mathbf{x}_u &= \sum_{l=1}^{k-1} \mathbf{x}_{u_l} + (n - k - 4)\mathbf{x}_u
		\implies \sum_{l=1}^{k-1} \mathbf{x}_{u_l} = (\rho - n + k + 4)\mathbf{x}_u.
	\end{align*}
	
	Consequently, the eigenequations at $v_1$ and $u_{k-1}$ can be written as
	\begin{align}
		\rho \mathbf{x}_{v_1}
		&= \sum_{l=1}^{k-1} \mathbf{x}_{u_l} + \sum_{i=2}^4 \mathbf{x}_{v_i}
		= (\rho-n+k+4)\mathbf{x}_u + 3\mathbf{x}_v, \label{eq:Fprime-v1}\\
		\rho \mathbf{x}_{u_{k-1}}
		&= \mathbf{x}_{v_1} + \sum_{l=1}^{k-2} \mathbf{x}_{u_l} + \sum_{i=k}^{n-4} \mathbf{x}_{u_i}
		= (\rho-2)\mathbf{x}_v + (n-k-3)\mathbf{x}_u. \label{eq:Fprime-uk-1}
	\end{align}
	
	Subtracting the eigenequation at $v_1$ from that at $u_{k-1}$ gives
	\[
	\rho(\mathbf{x}_{u_{k-1}}-\mathbf{x}_{v_1})
	=
	\mathbf{x}_{v_1}-\mathbf{x}_{u_{k-1}} + (n-k-3)\mathbf{x}_u - 3\mathbf{x}_v,
	\]
	and hence
	\[
	(\rho+1)(\mathbf{x}_{u_{k-1}}-\mathbf{x}_{v_1})=(n-k-3)\mathbf{x}_u-3\mathbf{x}_v.
	\]
	
	Assume to the contrary that $\mathbf{x}_{u_{k-1}}<\mathbf{x}_{v_1}$. Then, by the above identity and $\rho+1>0$, we obtain $(n-k-3)\mathbf{x}_u-3\mathbf{x}_v<0$. Hence,

	\begin{equation}\label{eq:xu_upper_general}
		\mathbf{x}_u<\frac{3}{n-k-3}\mathbf{x}_v.
	\end{equation}
		
	On the other hand, subtracting \eqref{eq:Fprime-v1} from \eqref{eq:Fprime-uk-1}, we obtain
	\[
	\rho(\mathbf{x}_{u_{k-1}}-\mathbf{x}_{v_1})
	=
	(\rho-5)\mathbf{x}_v + (2n-2k-7-\rho)\mathbf{x}_u.
	\]
	By our assumption $\mathbf{x}_{u_{k-1}}<\mathbf{x}_{v_1}$, we obtain $(\rho-5)\mathbf{x}_v+(2n-2k-7-\rho)\mathbf{x}_u <0$. Notice that $K_{n-4}$ is a proper subgraph of $F'_{n,k}$, so by Lemma \ref{lem:spectral_subgraph}, we have $\rho > \rho(K_{n-4})=n-5$. Since $n\ge 3k+2 \ge 11$, it follows that $5-\rho <5-(n-5)=10-n\le -1<0$. Thus,
	\[
	(2n-2k-7-\rho)\mathbf{x}_u<(5-\rho)\mathbf{x}_v < 0.
	\]
	Since $\mathbf{x}_u>0$, we have $2n-2k-7-\rho <0$, and hence
	\begin{equation}\label{eq:xu_lower_general}
		\mathbf{x}_u>\frac{5-\rho}{2n-2k-7-\rho}\mathbf{x}_v.
	\end{equation}
	
	Combining \eqref{eq:xu_upper_general} and \eqref{eq:xu_lower_general}, we obtain
	\[
	\frac{5-\rho}{2n-2k-7-\rho}<\frac{3}{n-k-3}.
	\]
	However,
	\[
	\frac{5-\rho}{2n-2k-7-\rho}-\frac{3}{n-k-3}
	=
	\frac{(k+6-n)(\rho+1)}{(2n-2k-7-\rho)(n-k-3)}.
	\]
	Here $k+6-n<0$ because $n \ge 3k+2$ and $k \ge 3$, while $\rho+1>0$, $2n-2k-7-\rho<0$ and $n-k-3>0$. Hence the right-hand side is positive, a contradiction.
	
	Therefore, $\mathbf{x}_{u_{k-1}} \ge \mathbf{x}_{v_1}$. By Lemma \ref{lem:perron_edge_shifting}, we conclude that $\rho(F_{n,k})>\rho(F'_{n,k})$.

	This completes the proof.
\end{proof}

\begin{lemma}\label{(k+1)_edge_connected}
	Let $k \ge 3$ and $n \ge 3k+2$. Let $G$ be a graph with vertex partition
	$V(G)=U\cup T$, where $|U|=n-4$ and $|T|=4$. If $G[U]$ is a clique and
	$d_G(v)\ge k+1$ for every $v\in T$, then $G$ is $(k+1)$-edge-connected.
\end{lemma}

\begin{proof}
	Let $S\subset V(G)$ be a nonempty proper subset, and denote $\bar S=V(G)\setminus S$.
	It suffices to show that $e_G(S,\bar S)\ge k+1$.
	
	If $v\in T$, then $d_G(v)\ge k+1$. If $u\in U$, then $d_G(u)\ge n-5$ since $G[U]$ is a clique, and $n-5\ge 3k-3>k+1$. Hence $\delta(G)\ge k+1$. In particular, if $|S|=1$ or $|\bar S|=1$, then $e_G(S,\bar S)\ge \delta(G)\ge k+1$.
	
	Thus, we may assume that $|S|\ge 2$ and $|\bar S|\ge 2$. There are two cases.
	
	Suppose first that $U\cap S\neq\emptyset$ and $U\cap \bar S\neq\emptyset$. Let
	$x=|U\cap S|$ and $y=|U\cap \bar S|$. Then $x,y\ge 1$ and $x+y=n-4$, so
	$1\le x\le n-5$. Since $G[U]$ is a clique, every edge between $U\cap S$ and
	$U\cap \bar S$ is present. Therefore
	\[
	e_G(S,\bar S)\ge xy=x(n-4-x)\ge n-5\ge 3k-3>k+1.
	\]
	
	Now suppose that $U\subseteq S$ or $U\subseteq \bar S$. By symmetry, it suffices to consider the case $U\subseteq S$. Then $\bar S\subseteq T$ and $2\le |\bar S|\le 4$. Let $r=|\bar S|$. Since each vertex of $\bar S$ has degree at least $k+1$ and $G[\bar S]$ has at most $\binom r2$ edges, we have
	\[
	e_G(S,\bar S)=\sum_{v\in\bar S} d_G(v)-2e_G(\bar S)
	\ge r(k+1)-2\binom r2 = r(k-r+2).
	\]
	If $r=2$, then $e_G(S,\bar S)\ge 2k$; if $r=3$, then $e_G(S,\bar S)\ge 3k-3$; and if $r=4$, then $e_G(S,\bar S)\ge 4k-8$. Since $k\ge 3$, all three lower bounds are at least $k+1$.
	
	Therefore $e_G(S,\bar S)\ge k+1$ for every nontrivial edge cut $E_G(S,\bar S)$, and hence $G$ is $(k+1)$-edge-connected.
\end{proof}

By Lemma \ref{(k+1)_edge_connected}, both $F_{n,k}$ and $F'_{n,k}$ are $(k+1)$-edge-connected. We are now ready to prove Theorem \ref{thm:c1}.

\begin{proof}[Proof of Theorem \ref{thm:c1}]
	By Theorem \ref{thm:k+c} with $c=1$, it remains to show that $F_{n,k}$ is the unique graph in $\mathcal{H}_{n,k}^{1}$ with maximum spectral radius. Indeed, when $c=1$, we have $n\ge (c+2)k-c^2+c+2=3k+2$, and every graph $H\in \mathcal{H}_{n,k}^{1}$ has a partition $V(H)=U\cup T$ such that $|U|=n-4$, $|T|=4$, $H[U]\cong K_{n-4}$, $5\le e_H(T)\le 6$ and $e_H(T,U)=4k-1-e_H(T)$.
	
	Let $G$ be a graph in $\mathcal{H}_{n,k}^{1}$ with maximum spectral radius. Then there exists a partition $V(G)=U\cup T$ such that $|U|=n-4$, $|T|=4$, $G[U]\cong K_{n-4}$, $5\le e_G(T)\le 6$ and $e_G(T,U)=4k-1-e_G(T)$. Let $T=\{v_1,v_2,v_3,v_4\}$ and $U=\{u_1,u_2,\ldots,u_{n-4}\}$. Let $\mathbf{x}$ be the Perron vector of $A(G)$, and assume, without loss of generality, that $\mathbf{x}_{u_1}\ge \mathbf{x}_{u_2}\ge \cdots \ge \mathbf{x}_{u_{n-4}}$.
	
	\noindent \textbf{Claim 1.} \textit{For any vertex $v\in T$, if $d_U(v)=s$, then $N_U(v)=\{u_1,u_2,\ldots,u_s\}$.}
	
	\noindent \textit{Proof.} Suppose that there exist $1\le r<j\le n-4$ such that $vu_j\in E(G)$ and $vu_r\notin E(G)$. Let $G'=G-vu_j+vu_r$. Then $G'[U]\cong K_{n-4}$, $e_{G'}(T)=e_G(T)$ and $e_{G'}(T,U)=e_G(T,U)$. Moreover, the degrees of the vertices in $T$ are unchanged, so $G'$ is still $(k+1)$-edge-connected by Lemma \ref{(k+1)_edge_connected}. Hence $G'\in \mathcal{H}_{n,k}^{1}$. Since $\mathbf{x}_{u_r}\ge \mathbf{x}_{u_j}$, Lemma \ref{lem:perron_edge_shifting} yields $\rho(G')>\rho(G)$, contradicting the maximality of $\rho(G)$. Thus, the claim holds. \qed
	
	We now discuss the two cases $e_G(T)=6$ and $e_G(T)=5$, respectively.
	
	\noindent \textbf{Case 1.} $e_G(T)=6$.
	
	In this case, $G[T]\cong K_4$ and $e_G(T,U)=4k-7$. Since $\delta(G)\ge k+1$, each vertex of $T$ sends at least $k-2$ edges to $U$. Hence exactly one vertex of $T$ sends $k-1$ edges to $U$, and each of the other three vertices sends $k-2$ edges to $U$. Without loss of generality, assume that $d_U(v_1)=k-1$ and $d_U(v_i)=k-2$ for $i\in\{2,3,4\}$. By Claim 1, we have
	\[
	E_G(T,U)=\{v_1u_\ell:1\le \ell\le k-1\}\cup \{v_i u_\ell:i\in\{2,3,4\},\,1\le \ell\le k-2\}.
	\]
	Hence $G\cong F'_{n,k}$.
	
	\noindent \textbf{Case 2.} $e_G(T)=5$.
	
	Without loss of generality, assume that $G[T]\cong K_4-v_1v_2$. Then $e_G(T,U)=4k-6$. Since $d_T(v_1)=d_T(v_2)=2$ and $d_T(v_3)=d_T(v_4)=3$, the condition $\delta(G)\ge k+1$ implies that $v_1$ and $v_2$ each send at least $k-1$ edges to $U$, while $v_3$ and $v_4$ each send at least $k-2$ edges to $U$. Since the total number of such edges is exactly $4k-6$, it follows that $d_U(v_1)=d_U(v_2)=k-1$ and $d_U(v_3)=d_U(v_4)=k-2$. By Claim 1, we have
	\[
	E_G(T,U)=\{v_i u_\ell:i\in\{1,2\},\,1\le \ell\le k-1\}\cup \{v_i u_\ell:i\in\{3,4\},\,1\le \ell\le k-2\}.
	\]
	Hence $G\cong F_{n,k}$.
	
	From the above analysis, the graph $G$ is isomorphic to either $F_{n,k}$ or $F'_{n,k}$. By Lemma \ref{lem:F2_F4_compare}, we have $\rho(F_{n,k})>\rho(F'_{n,k})$. Since $G$ has maximum spectral radius in $\mathcal{H}_{n,k}^{1}$, we conclude that $G\cong F_{n,k}$. Therefore, $F_{n,k}$ is the unique graph in $\mathcal{H}_{n,k}^{1}$ with maximum spectral radius. This completes the proof.
\end{proof}

\section{Concluding remarks}

We end this paper with some remarks. We first recall the result of Zhang and Fan \cite{Zhang2026} on $k$-edge-connected graphs.

\begin{theorem}[\cite{Zhang2026}]\label{thm:c0}
	Let $k\ge 2$, and let $G$ be a $k$-edge-connected graph of order $n\ge 2k+3$. If $\rho(G)\ge \rho\left(K_{k-1}\vee (K_{n-k-1}\cup K_2)\right)$, then $G$ contains $k$ edge-disjoint spanning trees unless $G\cong K_{k-1}\vee (K_{n-k-1}\cup K_2)$.
\end{theorem}

The extremal graph $K_{k-1}\vee (K_{n-k-1}\cup K_2)$ in Theorem \ref{thm:c0} can be viewed as a large clique $K_{n-2}$ together with two vertices outside this clique. If we denote these two vertices by $T=\{v_1,v_2\}$, then $e(T)=1$, and there are exactly $2k-2$ edges joining $T$ and the clique $K_{n-2}$.

It is easy to see that if we set $c=0$ in Theorem \ref{thm:k+c}, then the resulting extremal structure is consistent with Theorem \ref{thm:c0}. The difference lies in the lower bound on $n$. In Theorem \ref{thm:k+c}, the condition $n\ge (c+2)k-c^2+c+2$ becomes $n\ge 2k+2$ when $c=0$. However, in the proof of Claim 1, the corresponding function $\phi(n)$ is increasing for $n\ge 2k+2$, while $\phi(2k+2)=1-k<0$ and $\phi(2k+3)=1>0$. Hence, in the case $c=0$, we need the stronger condition $n\ge 2k+3$ to obtain the required contradiction. In other words, if the lower bound of $n$ in Theorem \ref{thm:k+c} is replaced by $n\ge (c+2)k-c^2+c+3$, then Theorem \ref{thm:k+c} holds for all $0\le c\le k-2$.

Recall Theorem \ref{thm:k-1}. Let $k=2$, we obtain $t_1=5$, $e_G(T)=6$ and $e_G(T,U)=3$. This matches the exact edge distribution in the unique extremal graph given in Theorem \ref{thm1}.

It is foreseeable that determining the concrete extremal graph becomes much harder as $c$ becomes larger. The number of possible candidate graphs increases and comparing their spectral radii is not easy. For $2\le c\le k-1$, it remains interesting to determine concrete extremal graphs corresponding to the tight spectral radius condition for $(k+c)$-edge-connected graphs to contain $k$ edge-disjoint spanning trees. We leave this question for future work.

\section*{Declaration of competing interest}
\quad\quad The authors declare that they have no known competing financial interests or personal relationships that could have appeared to influence the work reported in this paper.

\section*{Data availability}
\quad\quad No data was used for the research described in the article.

\end{document}